\documentclass[a4paper]{jnmp}
\usepackage{amsmath}
\usepackage{amsthm}
\usepackage{amssymb}

\newcommand{\abs}[1]{\vert #1 \vert}  
\newcommand{\ov}{\overline}
\newcommand{\gr}{\mathrm{grad}}
\newcommand{\ii}{\mathrm{i}}

\newcommand{\dif}{\mathrm{d}}

\newcommand{\di}{\mathrm{div}}
\newcommand{\tr}{\mathrm{trace}}
\newcommand{\CC}{\mathbb{C}}
\newcommand{\Cc}{\mathfrak{C}}

\newcommand{\HH}{\mathcal{H}}
\newcommand{\VV}{\mathcal{V}}

\newcommand{\RR}{\mathbb{R}}

\newcommand{\Ss}{\mathbb{S}}
\newcommand{\Mm}{\mathbb{M}}
\newcommand{\Tt}{\texttt{T}}

\JNMPnumberwithin{equation}{section}
\resetfootnoterule

\newtheorem*{acknowledgments}{Acknowledgments}

\newtheorem{te}{Theorem}[section]
\newtheorem{pr}{Proposition}[section]
\newtheorem{co}{Corollary}[section]
\newtheorem{lm}{Lemma}[section]
\theoremstyle{definition}
\newtheorem{de}{Definition}[section]    
\newtheorem{re}{Remark}[section]
\newtheorem{ex}{Example}[section]

\begin{document}

\renewcommand{\evenhead}{R. Slobodeanu}
\renewcommand{\oddhead}{Perfect fluids from high power sigma-models}

\thispagestyle{empty}



\Name{Perfect fluids from high power sigma-models}

\label{firstpage}

\Author{Radu Slobodeanu}

\Address{Department of Theoretical Physics and Mathematics, University of Bucharest,\\
P.O. Box Mg-11, RO--077125 Bucharest-M\u agurele, Romania. \\
E-mail: \emph{radualexandru.slobodeanu@g.unibuc.ro}}


\begin{abstract}
\noindent Certain solutions of a sextic sigma-model Lagrangian reminiscent of Skyrme model correspond to perfect fluids with stiff matter equation of state. We analyse from a differential geometric perspective this correspondence extended to general barotropic fluids. 
\end{abstract}



\section{Introduction}
Relativistic hydrodynamics is an important theoretical tool in heavy-ion physics, astrophysics, and cosmology. In a first stage, it provides us with a good approximation in terms of \textit{perfect fluids} for many physical systems, from quark-gluon plasma to gaseous interior of a collapsing star. A perfect fluid is defined \cite{wei} as having at each point a velocity $\vec{v}$, such that an observer moving with this velocity sees the fluid around him as isotropic. For a standard introduction in this subject see e.g. \cite{choq, mis, wei}.

The fact that perfect fluid dynamics admits a Lagrangian formulation is widely known. In the approach taken in \cite{bei, cristo, langl, ratta, smi}, the Lagrangian is given in terms of a "dual" variable $\varphi$ which is a submersion / field  defined on the spacetime such that the fluid velocity four-vector spans its vertical foliation. In this (locally) dual picture the Euler equations for a stiff fluid come by variational principle from a sextic Lagrangian in the derivatives of $\varphi$, and the conservation of energy along flow lines is automatically satisfied by construction. It is interesting to notice that this sextic Lagrangian is exactly the extension of the Skyrme model discussed in \cite{ada, piet}. In fact the equations for all perfect fluids with cosmological equation of state come in a similar way by raising to a proper power this sextic Lagrangian. 
In this paper, we detail these constructions in the standard (coordinate-free) differential geometric setup and then we address the following topics: the geometric interpretation of the Euler equations for the fluid (Euler-Lagrange equations for the field), the effect of (bi)conformal change of metrics, the irrotational case, the shear-free case in relation with the harmonic morphisms theory and the coupling with gravity.  We illustrate these issues with detailed examples that cover most of the presently known exact solutions for perfect fluid equations. The solutions are re-derived in this new context via the reduction by symmetry method that simplify Euler equations to an ODE. Their natural derivation and the fact that they automatically satisfy the conservation of energy along the flow are two primary advantages of the present approach that may further contribute to uncover other \textit{exact} solutions yielding a new insight into the fluid dynamics \cite{nagy}.

Our initial motivation was to bring the relativistic fluid problem in the intensively studied area of geometric variational methods and analysis of critical mappings between manifolds (mainly developed in the Riemannian case, see e.g. \cite{ud, ura}), but the paper can also be connected with the kinetically driven inflation models \cite{muk, arro, di} or gravitating skyrmions \cite{pie}.

\section{Perfect fluids} 

Throughout the paper $(M^m, g)$ will denote an $m$-dimensional Lorentz manifold. In particular, the
$m$-dimensional \textit{Minkowski space} $\Mm^m = \RR^{m}_{1}$
is defined to be $\RR^{m}$ endowed with the metric of signature
$(1, m - 1)$ given in standard coordinates $(x_1, x_2, . . . , x_m)$ by $g = -\dif x^{2}_{m} + \dif x^{2}_{1} + . . . + \dif x^{2}_{m-1}$ (so $x_m$ denote the time and $x_i$, $i=1,..., m-1$ the spatial directions). By ($m$-dimensional) \textit{spacetime} we shall mean a connected time-oriented Lorentz manifold.

Recall that a smooth distribution $\VV$ of dimension $q$ on $M$ is a smooth rank $q$ subbundle of the tangent bundle $TM$. Assuming that $\VV$ is non-degenerate, the metric $g$ gives rise to a splitting $TM = \VV \oplus \HH$ and one can define the \textit{mean curvature of} $\VV$ by
\begin{equation}
\mu^{\VV}=\frac{1}{q}\sum_{i=1}^{q}\varepsilon_i \left(\nabla_{e_i}e_i\right)^{\HH} ,
\end{equation}
where $\nabla$ is the Levi-Civita connection, $\{e_i \}$ is a (local moving) frame for $\VV$, $\varepsilon_i = g(e_i, e_i) = \pm 1$ and $X^\HH$ denotes the orthogonal projection on $\HH$ of a vector field $X$ on $M$. If $\VV$ is closed under the Lie bracket, then it is called \textit{integrable} distribution and, by Frobenius' theorem, it gives rise to a smooth foliation of $M$. The foliation / distribution $\VV$ is said to be \textit{extremal} if $\mu^{\VV}=0$ (the terminology is inspired by the (local) "area" maximization property of zero mean curvature spacelike hypersurfaces, among all nearby hypersurfaces having the same boundary \cite{fra}. In the Riemannian case the same condition assures the minimality in a similar variational problem, see e.g. \cite{ura}).

\medskip

Now let us recall a slightly simplified picture of perfect fluids, following \cite[Chap. 12]{one}.
\begin{de}\label{def}
A \emph{perfect fluid} on a spacetime $M$ is a triple $(U, \rho, p)$ such that:
\begin{enumerate}
\item[(i)] $U$ is a timelike future-pointing unit vector field on $M$, called the \emph{flow vector field}.

\item[(ii)] $\rho, p:M\to \RR$ are the \emph{mass (energy) density} and \emph{pressure}, respectively.

\item[(iii)] $\di \Tt =0$, where $\Tt$ is the \emph{stress-energy tensor} given by:
\begin{equation}\label{perfstress}
\Tt=p \, g + (p + \rho)\omega \otimes \omega ,
\end{equation}
with $\omega=U^{\flat}$ the 1-form metrically equivalent to $U$.
\end{enumerate}
\end{de}

\begin{pr}\label{perfeq}
If $(U, \rho, p)$ is a perfect fluid, then the following relations hold:

\medskip
$(i)$ \ $(\rho + p)\di U + U(\rho)=0$ (\emph{conservation of energy along flow lines})

\medskip
$(ii)$ \ $(\rho+p)\nabla_{U}U + \gr^{\perp}p=0$ (\emph{Euler equations}),

\medskip
\noindent where $\gr^{\perp}p$ (the \emph{spatial pressure gradient}) denotes the component orthogonal to $U$.
\end{pr}

In what follows we restrict ourself to \textit{barotropic fluids}, i.e. we assume the \textit{equation of state} (EoS): $\rho=\rho(p)$. In this case, a third equation governing the fluid (the \textit{particle number conservation}) decouples and therefore it is enough to work with the above definition.

In this case it is useful to introduce the \textit{index of the fluid} \cite{choq} by:
$\displaystyle f(p) := \mathrm{exp} \int \frac{\dif p}{\rho(p) + p}$.
For reasons that will become clear below, we shall denote the foliation defined by $U$ with $\VV$ calling it \textit{vertical} and its complementary with $\HH$, calling it \textit{horizontal} distribution. Their mean curvatures are in this case given by:
\begin{equation}\label{mcurv}
\mu^{\VV} = -\nabla_{U}U, \qquad \mu^{\HH} = \frac{\di U}{m-1}U.
\end{equation}
The fluid's equations above rewrite as follows:
\begin{equation}\label{curvgrad}
\mu^{\HH} = - \gr^{\VV}(\ln (f^{\prime})^{\frac{1}{m-1}}); \qquad
\mu^{\VV} = \gr^{\HH}(\ln f).
\end{equation} 

\begin{ex}[Radiation case]
For instance, considering radiation's EoS, $\rho=3p$ and $m=4$, we can resume fluid's equations in
$$\mu^{\HH}+\mu^{\VV} = \gr(\ln f),$$ with $f=p^{1/4}$. The fact that $\mu^{\HH}+\mu^{\VV}$ is of gradient type means, geometrically, that both $\VV$ and $\HH$ are \textit{extremal} (i.e. $\mu^{\HH}=0$ and $\mu^{\VV}=0$) with respect to the metric $\ov g = f^2 g$ conformal to the spacetime metric. In particular, the flow lines are geodesics of $\ov g$, cf. \cite[Theorem 10.1]{choq}.
\end{ex}

\section{Perfect fluids from sigma-models}

\subsection{The Cauchy-Green tensor}
Let $\varphi: (M^{m}, g) \to (N^{n}, h)$ a $C^1$ map defined on a spacetime with values in a Riemannian manifold.
The pullback metric $\varphi^{*}h$ is a 2-covariant tensor field on $M$. If we define $\dif \varphi^{t}:TN \rightarrow TM$ the formal adjoint of $\dif \varphi$ with respect to $g$ and $h$ in the usual way, then we can think of $\varphi^{*}h$ as an endomorphism $\Cc _{\varphi} = \dif \varphi^{t} \circ \dif \varphi: TM \rightarrow TM$ called the \textit{Cauchy-Green (strain) tensor} of $\varphi$, analogously to the case of deformations in non-linear elasticity.  It is self-adjoint by construction, but due to the Lorentzian signature of the metric, it may have complex eigenvalues. We shall eliminate this possibility by asking our maps to have timelike (so non-degenerate) vertical distribution $\VV=\ker \dif \varphi$.  

For the convenience of the reader we present the following basic linear algebraic fact by specifying the dimensions to the case of interest for us (for general facts about self-adjoint linear operators on Lorentz vector spaces, see \cite[p. 243, 261]{one}).
\begin{lm} \label{alg}
Let $F: V \to W$ be a linear operator of rank $3$ from a Lorentz vector space $V \approx \RR_{1}^{4}$ to an Euclidean vector space $W \approx \RR^{3}$.
On $V$, consider the self-adjoint operator $\Cc := F^{t} \circ F$, where $F^t: W \to V$ is defined by: $\langle Fv, w\rangle _{Euclid} = \langle v, F^t w\rangle _{Minkowski}$, $\forall v \in V, w \in W$. 

If $\langle v, v \rangle _{Minkowski} < 0, \forall v \in \ker F$, then the eigenvalues of $\Cc$ are (real and) non-negative and $\Cc$ can be diagonalized in an orthonormal basis (of eigenvectors).
\end{lm}

\begin{proof}
Let $\mathcal{F} = \left(f_{ij} \right)_{i=1,2,3}^{j=1,2,3,4}$ be the associated (real) matrix of $F$, with respect to some orthonormal frames. Then the associated matrix of the adjoint  $F^t$ is the \emph{pseudo-transposed} matrix
$\mathcal{F}^{\diamond}=\left(\begin{array}{ccc}
-f_{11} & -f_{21} & -f_{31}\\
f_{12} & f_{22} & f_{32} \\
f_{13} & f_{23} & f_{33} \\
f_{14} & f_{24} & f_{34}
\end{array} \right)$. The associated matrix of the composition is a \textit{pseudo-symmetric} matrix of the form
$\mathcal{F}^{\diamond}\cdot \mathcal{F}=\left(\begin{array}{cc}
\alpha & \varpi\\
-\varpi^t & A \\
\end{array} \right)$, where $\alpha \in \RR$, $A$ is a symmetric matrix and $\varpi$ a row 3-vector. Now let us apply a standard argument similar to the Euclidean case. Take $Z =(z_i)_{1\leq i \leq 4} \in \CC^4$ a solution of the system $\mathcal{S}: (\mathcal{F}^{\diamond}\cdot \mathcal{F} - \lambda \mathbb{I}_4)Z=0$ and multiply it on the left with the diagonal matrix $\mathrm{diag}(-\ov z_1, \ov z_2, \ov z_3, \ov z_4)$. Summing the new equations provides us with a linear equation in $\lambda$:  \ $\lambda(- z_1 \ov z_1 + \sum_{i=2}^{4} z_i \ov z_i) +$ (real terms) $=0$. So $\lambda$ will be real, unless we have some complex lightlike solution of $\mathcal{S}$. But if it exists such solution $Z = X+ \mathrm{i} Y$, then $\langle X , X \rangle + \langle Y , Y \rangle =0$ and $X$ and $Y$ are in the 1-dimensional kernel of $F$. So $X$ and $Y$ must be colinear and lightlike, that contradicts our hypothesis.

The proof of the second statement is similar to the Euclidean case.
\end{proof}

So, at a point $x$ of $M$, if $\ker \dif \varphi_x$ is timelike, we saw that there exists an orthonormal basis of eigenvectors of $\Cc_{\varphi}$ for $T_x M$. Extending \cite[Lemma 2.3]{pant} in this context, we get a \textit{local} orthonormal frame of eigenvector fields, around any point where the timelike condition above is satisfied (i.e. $\Cc_{\varphi}$ can be \textit{consistently diagonalized}).

In general, let $\Lambda_{1}$, $\Lambda_{2}$, ..., $\Lambda_{r}$ and $\Lambda_{r+1} = ... = \Lambda_{m}=0$
be the eigenvalues of $\Cc_{\varphi}$, where $r := \mathrm{rank} (\dif \varphi)$ everywhere. Whenever they are all real, the corresponding elementary symmetric functions, $\sigma_k(\varphi)$, will play a special role, due to the fact that
$$
\sigma_1(\varphi)= \abs{\dif \varphi}^2; \qquad
\sigma_k(\varphi) = \abs{\wedge^k \dif \varphi}^2; \quad . . . \ ;
\quad \sigma_m(\varphi)= \mathrm{det}(\varphi^* h),
$$
where $\wedge^k \dif \varphi$ denotes the induced map on $k$-vectors and in the right-hand side of every identity we have the usual Hilbert Schmidt square norm induced by $g$ and $h$ (that may be negative, despite the squared notation).

\subsection{High power Lagrangians and their associated stress-energy tensor}

For a $C^1$ map $\varphi: (M, g) \to (N, h)$ between (semi-) Riemannian manifolds let us recall some Lagrangians and their action integral $\mathcal{E}$ (that will be called \textit{energy} to recover the Riemannian geometric terminology). The associated stress-energy tensor $S(\varphi)$ is defined by the first variation formula with respect to the domain metric variations: 
\begin{equation*}
\frac{d}{d u} \Bigl\lvert _{u=0}\mathcal{E}(\varphi, g_u)=\int_M \langle S(\varphi),\ \delta g \rangle \nu_g.
\end{equation*}

\begin{enumerate}
\item[$\bullet$] (cf. \cite{bg}) $k$-\textbf{energy} of $\varphi$: $\mathcal{E}_k(\varphi) = \int_{M} e_k(\varphi)\nu_{g}$, where $e_k(\varphi)=\frac{1}{k}\abs{\dif \varphi}^k$, \ $k \geq 2$.

$k$-\textbf{stress-energy tensor}:
\begin{equation}
S_k({\varphi}):= e_k(\varphi)g - \abs{\dif \varphi}^{k-2}\varphi^* h .
\end{equation}
Adapting the proof in the Riemannian case, we easily get the identity:
\begin{equation}\label{ide}
h(\tau_k(\varphi), \dif \varphi)=-\di S_k({\varphi}) ,
\end{equation}
where $\tau_k(\varphi)=\tr_{g} \nabla (\abs{\dif \varphi}^{k-2}\dif \varphi)$ is the \textit{k-tension field} of $\varphi$. For submersions, \eqref{ide} tells us that $S_k({\varphi})$ is divergence free if and only if $\varphi$ is $k$-\textit{harmonic}, i.e. a stationary point of $\mathcal{E}_k(\cdot)$ with respect to smooth variations of $\varphi$ on any compact domain $K$ ($\frac{d}{d s} \vert _{s=0}\mathcal{E}_{k}(\varphi_s)=0$). A $2$-harmonic map is called simply \textit{harmonic map} or \textit{wave map} if $M$ is Lorentzian.

\item[$\bullet$] (cf. \cite{cri}) $\sigma_k$-\textbf{energy} of $\varphi$: $\mathcal{E}_{\sigma_k}(\varphi)=
\frac{1}{2}\int_{M}\sigma_k(\varphi)\nu_{g}$, where $\sigma_k(\varphi)=\abs{\wedge^k \dif \varphi}^2$, \ $k\geq 1$.

$\sigma_k$-\textbf{stress-energy tensor}: 
\begin{equation}\label{stres3}
S_{\sigma_k}(\varphi)=\frac{1}{2}\sigma_k(\varphi)g - \varphi^* h \circ \chi_{k-1}(\varphi),
\end{equation}  
where $\chi_q$ denotes the $q^{th}$ Newton tensor. Again, for submersions we can prove that $S_{\sigma_k}({\varphi})$ is divergence free if and only if $\varphi$ is $\sigma_k$-\textit{critical}, i.e. $\frac{d}{d s} \vert _{s=0}\mathcal{E}_{\sigma_k}(\varphi_s)=0$ on any compact domain $K$, or equivalently $\tau_{\sigma_k}(\varphi)=0$, where $\tau_{\sigma_k}(\varphi)=\tr_{g} \nabla (\dif \varphi \circ \chi_{k-1}(\varphi))$ is called the $\sigma_k$-\textit{tension field} of $\varphi$.
\end{enumerate}

Notice that $\frac{1}{2}\sigma_1(\varphi) = e_2(\varphi)$ is the \textit{standard kinetic term} and $\sigma_2(\varphi)$ coincides with the (self-interacting) \textit{Skyrme term} (for fields with target $N$). The stress-energy tensors of these two cases appear explicitly in \cite{gib} where one shows that they satisfy the dominant energy condition
\footnote{In our setting the stress-energy has the opposite sign with respect to the physics literature where it is defined as a variational derivative for the Lagrangian density relative to the \textit{inverse} metric.}. 

It is worth to emphasize that the identities (\ref{ide}) (consequences of the diffeomorphism invariance of the action) implies that the stress-energy tensor is always conserved along vertical directions (i.e. $[\di S(\varphi)](V)=0, \forall V \in \mathrm{Ker} \, \dif \varphi$) for \textit{any} map $\varphi$, not necessarily for critical maps. In particular, for $\sigma_3$-energy we have the following general result, whose independent proof can be find in \cite[Prop. 3]{slub}.

\begin{lm}\label{diu}
Consider a submersion $\varphi : (M^{4}, g) \to (N^3, h)$ between (semi) Riemannian manifolds such that $\sigma_3(\varphi)=\mathrm{det}(\varphi^*h) \geq 0$ on an open set. Let $U$ be an unit vertical vector field on this open set. Then the following identity holds:
\begin{equation}\label{titi}
\di U + U(\ln \sqrt{\sigma_3(\varphi)})=0.
\end{equation}
\end{lm}

\subsection{The $\sigma_{3}$-critical field / stiff perfect fluid correspondence}

In the following we identify a first case when the stress-energy tensor of a $\sigma$-model is in the perfect fluid form (see also \cite{bei, cristo, langl}). This happens also for the stress-energy of a \textit{scalar} field (kinetic or higher power) Lagrangian with potential, but in this case $U$ is constrained to be irrotational, see e.g. \cite{muk, di}.

\begin{pr}\label{stif}
Let $\varphi : (M^{4}, g) \to (N^3, h)$ be a submersion from a Lorentzian manifold onto a Riemannian manifold. Suppose that the fibers are timelike and take $U$ the unit vector field that spans the vertical space $\VV = \ker \dif \varphi$. Denote by $\HH = \VV^{\perp}$ its complementary distribution and let $g = g^{\HH} - g^{\VV}$ be the induced splitting of the metric tensor, with $g^{\VV}=\omega \otimes \omega$, $\omega(X):= g(U, X)$, $\forall X \in \Gamma(TM)$.  Then the following statements hold good:

$(i)$ \ The eigenvalues of the Cauchy-Green tensor $\Cc_\varphi$ with respect to $g$ are all positive: $\Lambda_i=\lambda_{i}^{2}$, $i=1,2,3$, $\Lambda_4=0$.

$(ii)$  $\sigma_3$-stress-energy tensor of $\varphi$ has the perfect fluid form 
\begin{equation}\label{s3pf}
-2S_{\sigma_3}(\varphi) = \Tt = p g^{\HH} + \rho g^{\VV},
\end{equation}
with the \emph{stiff matter} EoS: $p=\rho=\lambda_{1}^{2}\lambda_{2}^{2}\lambda_{3}^{2}$.

$(iii)$ \ $\varphi$ is $\sigma_3$-critical if and only if
\begin{equation}\label{s3}
-\gr^{\HH}(\ln \lambda_{1}\lambda_{2}\lambda_{3}) + \mu^{\VV}=0.
\end{equation} 
that is if and only if $(U, \rho, p)$ satisfies the relativistic Euler equation for the EoS \\ $p=\rho=\lambda_{1}^{2}\lambda_{2}^{2}\lambda_{3}^{2}$. 

$(iv)$ The following identity is true
\begin{equation}\label{s3+}
3\mu^{\HH}=\gr^{\VV}(\ln \lambda_{1}\lambda_{2}\lambda_{3}).
\end{equation} 
and it is equivalent to the conservation of energy along flow lines for $(U, \rho, p)$ with \\ $p=\rho=\lambda_{1}^{2}\lambda_{2}^{2}\lambda_{3}^{2}$.

$(v)$ The fluid equations taken together are equivalent to: 
\begin{equation*}
3\mu^{\HH} + \mu^{\VV} \quad \mathrm{is \ of \ gradient \ type}.
\end{equation*}
In particular after a biconformal rescaling of the metric both $\VV$ and $\HH$ become extremal.

$(vi)$ \ The equation (\ref{s3}) is invariant under biconformal changes of metric of the following type:
\begin{equation}\label{bico}
\ov g = \varsigma^{-2}g^{\HH} - \varsigma^{-6} g^{\VV}.
\end{equation}  
In particular, if $\varphi$ is $\sigma_3$-critical, then $(\varsigma^{3} U, \varsigma^{6}\rho, \varsigma^{6}p)$ is a stiff perfect fluid on $(M, \ov g)$.
\end{pr}

\begin{proof}
$(i)$ Consequence of Lemma \ref{alg}.

$(ii)$ \ Recall \cite{cri} that the second Newton tensor is given by:
\begin{equation*}
\chi_2(\varphi) = \sigma_2(\varphi)Id_{TM} - \Cc_{\varphi} \circ \chi_1(\varphi); \quad \chi_1(\varphi) = 2e_2(\varphi)Id_{TM} - \Cc_{\varphi},
\end{equation*}
where $\Cc_{\varphi}=\dif \varphi^t \circ \dif \varphi$ is the Cauchy-Green tensor. Let $\HH=\HH_1 \oplus \HH_2 \oplus \HH_3$ be the splitting of the horizontal distribution in eigenspaces of $\Cc_{\varphi}$. Take $X_1\in \Gamma(\HH_1)$ and let us compute the $\sigma_3$-stress-energy (\ref{stres3}) on it: 
\begin{equation*}
\begin{split}
&S_{\sigma_3}(\varphi)(X_1, Y)= \frac{1}{2}\sigma_3(\varphi)g(X_1, Y) - \varphi^* h(\chi_2(\varphi)(X_1), Y)\\
&= \frac{1}{2}\lambda_{1}^{2}\lambda_{2}^{2}\lambda_{3}^{2} g(X_1, Y) - \varphi^* h(\sigma_2(\varphi)X_1 -2e_2(\varphi)\Cc_{\varphi}(X_1)+\Cc_{\varphi}^{2}(X_1), Y)\\
&= \frac{1}{2}\lambda_{1}^{2}\lambda_{2}^{2}\lambda_{3}^{2} g(X_1, Y) - \varphi^* h(\left[\lambda_{1}^{2}\lambda_{2}^{2}+\lambda_{2}^{2}\lambda_{3}^{2}+
\lambda_{3}^{2}\lambda_{1}^{2} - (\lambda_{1}^{2}+\lambda_{2}^{2}+\lambda_{3}^{2})\lambda_{1}^{2}+
\lambda_{1}^{4}\right]X_1, Y)\\
&= \lambda_{2}^{2}\lambda_{3}^{2} \left(\frac{\lambda_{1}^{2}}{2}g^{\HH}(X_1, Y) - \varphi^* h(X_1, Y)\right)
=-\frac{1}{2}\lambda_{1}^{2}\lambda_{2}^{2}\lambda_{3}^{2} g^{\HH}(X_1, Y).
\end{split}
\end{equation*}

A similar identity holds for the other two eigenspaces. 

Since for any $V \in \Gamma(\VV)$ we have
\begin{equation*}
\begin{split}
S_{\sigma_3}(\varphi)(V, Y) &= \frac{1}{2}\sigma_3(\varphi)g(V, Y) - \varphi^* h(\chi_2(\varphi)(V), Y)= -\frac{1}{2}\lambda_{1}^{2}\lambda_{2}^{2}\lambda_{3}^{2}g^{\VV}(V, Y),
\end{split}
\end{equation*}
the conclusion follows. 

$(iii)$ Rewrite Euler equations from Proposition 2.1 in this context or simply take the divergence in (\ref{s3pf}). 

$(iv)$ Apply Lemma \ref{diu}. Then interpret the mean curvature of $\HH$ as we have done in (\ref{mcurv}).

$(v)$ Under a general biconformal change of metric 
$\ov g = \varsigma^{-2}g^{\HH} - \varrho^{-2} g^{\VV}$, the mean curvatures of horizontal and vertical distributions become:
\begin{equation}\label{curvbicon}
\begin{split}
\ov \mu^{\HH} &= \varrho^{2} \left[\mu^{\HH}+ \gr^{\VV}(\ln \varsigma)\right]; \\
\ov \mu^{\VV} &= \varsigma^{2} \left[\mu^{\VV}+\gr^{\HH}(\ln \varrho)\right].
\end{split}
\end{equation}  
As in our case both $\mu^{\HH}$ and $\mu^{\VV}$ are of gradient type, they can be set to zero by such a biconformal change of metric.

$(vi)$ Using the second equation in (\ref{curvbicon}), we can rewrite Equation (\ref{s3}) with respect to a biconformally related metric $\ov g$:
\begin{equation*}
\mu^{\VV} - \gr^{\HH}(\ln \lambda_{1}\lambda_{2}\lambda_{3}) + \gr^{\HH}(\ln \varsigma^{-3}\varrho)=0.
\end{equation*}
and the conclusion follows.
\end{proof}

\begin{re}\label{remstif}
$(a)$ The above result can be easily generalized for any dimension $m$ of $M$, considering $\sigma_{m-1}$-critical submersion with timelike one-dimensional fibers. In particular, the strongly coupled Faddeev-Niemi model \cite{fad} on $\Mm^3$ will have solutions interpretable as stiff perfect fluids in (2+1) dimensions.

\noindent Notice that the condition for fibers to be timelike assures the dominant energy condition for the high power sigma-model considered (recall our different sign convention for the stress-energy tensor). For a more general statement see \cite{wo}.

\medskip
$(b)$ A computation similar to $(ii)$ in the above proof, applied to the other possible energies $e_k(\varphi)$ and $\sigma_k(\varphi)$, shows us that the stress-energy tensor does not have the perfect fluid form (see also \cite{gib}).

\medskip
$(c)$ Conversely, given a timelike unit vector field $U$, by locally integrating it becomes tangent to the fibres of a submersion onto some Riemannian manifold and the above result applies.
So locally we are able to translate the problem for the triple $(U, p, \rho)$ in terms of $(\varphi, (N,h))$, once we have fixed a metric $g$ on the spacetime $M$.

\medskip
$(d)$ If $\lambda_{1}^{2}=\lambda_{2}^{2}=\lambda_{3}^{2}=\lambda^2$, then $\varphi$ is called \emph{horizontally conformal} \cite{ud} and the vertical unit vector $U$ is shear-free. In this case $(U, \rho, p)$ is a stiff perfect fluid if and only if $\varphi$ is 6-harmonic (morphism). This condition translates into the following analogue of the \textit{fundamental equations} for harmonic morphisms:
\begin{equation*}
(n-6)\gr^{\HH}(\ln \lambda) + (m-n)\mu^{\VV}=0,
\end{equation*}
where $m=4$, $n=3$ in our case (see the subsection 3.5 below). 

The particular case of semi-Riemannian submersions (i.e. $\lambda=1$) with extremal fibers ($\mu^{\VV}=0$) will describe then fluids with constant pressure and mass density.

\medskip
$(e)$ (Conformal change of the codomain metric). \ Let $\widetilde h = \ov{\nu}^{2}h$, where $\ov{\nu}:N \to (0, \infty)$ is a smooth function. Set $\nu = \ov{\nu} \circ \varphi$. Then the eigenvalues of $\varphi ^* \widetilde h$ are $\widetilde \lambda_{i}^{2}= (\nu \lambda_{i})^{2}$. The equations (\ref{s3+}) leave unchanged, while (\ref{s3}) become:\\
$-3\gr^{\HH}(\ln \nu)-\gr^{\HH}(\ln \lambda_{1}\lambda_{2}\lambda_{3}) + \mu^{\VV} = 0$.
\end{re}

Now let us illustrate the above discussion with two extensions of the well-known Hwa-Bjorken solution in (1+1) dimensions \cite{bjo} (for a more elaborate treatment of the reduction techniques involved below, see \cite[Chap.13]{ud}).

\begin{ex}[\textbf{Extended Hwa-Bjorken flow (I)}]\label{bjo1}
Let $\RR^3$ be the Euclidean space parametrized by spherical coordinates and regard $\mathbb{M}^4$ as $\RR \times \RR^3$. Consider $\varphi : \mathbb{M}^4 \to \RR^3$, an $SO(3)$-equivariant  submersion of the form:
\begin{equation}\label{bjsubm}
\left(x_4, \ r(\cos s , \ \sin s \cdot e^{i\theta})\right) \mapsto
\alpha(x_4, r) \left(\cos s , \ \sin s \cdot e^{i \theta}\right),
\end{equation}
depending on a function $\alpha : \RR \times [0, \infty) \to [0, \infty)$. Then the eigenvalues of $\Cc_\varphi$ are:
\begin{equation*}
\lambda_{1}^{2}=\left(\frac{\partial \alpha}{\partial r}\right)^2 - \left(\frac{\partial \alpha}{\partial x_4}\right)^2, \quad
\lambda_{2}^{2}=\lambda_{3}^{2}=\frac{\alpha ^2}{r^2}, \quad
\lambda_{4}^{2}=0
\end{equation*}

with the associated eigenvector fields forming an orthonormal frame: 
\begin{equation*}
X_1 = \frac{1}{\lambda_{1}}\left(\frac{\partial \alpha}{\partial x_4} \frac{\partial}{\partial x_4} - \frac{\partial \alpha}{\partial r} \frac{\partial}{\partial r} \right), \ 
X_2=\frac{1}{r}\frac{\partial}{\partial s}, \ X_3 = \frac{1}{r \sin s}\frac{\partial}{\partial \theta}, \ 
U = \frac{1}{\lambda_{1}}\left(\frac{\partial \alpha}{\partial r} \frac{\partial}{\partial x_4} -\frac{\partial \alpha}{\partial x_4} \frac{\partial}{\partial r}\right).
\end{equation*} 

Notice that the condition for $\left(\frac{\partial \alpha}{\partial r}\right)^2 - \left(\frac{\partial \alpha}{\partial x_4}\right)^2$ to be positive on an open set translates both the condition of timelike fibers and the positivity of the first eigenvalue.

Writing $U = \cosh \Omega(x_4, r) \cdot \frac{\partial}{\partial x_4} + \sinh \Omega(x_4, r) \cdot \frac{\partial}{\partial r}$, the equations (\ref{s3}) reduce to one single PDE:
\begin{equation}\label{HBe}
\cosh \Omega \left(\frac{\partial \Omega}{\partial x_4} + \frac{\partial (\ln \lambda_{1}\lambda_{2}\lambda_{3})}{\partial r} \right) + \sinh \Omega \left(\frac{\partial (\ln \lambda_{1} \lambda_{2} \lambda_{3})}{\partial x_4} + \frac{\partial \Omega}{\partial r} \right)=0.
\end{equation} 

Let us now restrict $\varphi$ to the interior of the forward lightcone $\mathcal{U} = \{ x \in \mathbb{M}^4 \vert \abs{x_4} > \abs{r} \}$ and consider the ansatz $\alpha(x_4, r)= f(\zeta)$, $\zeta:=r/x_4$. The common unit vertical vector for all maps in this ansatz is the (normalized) \textit{scale transformation}:
\begin{equation}\label{HBflow}
U = \frac{x_4}{\sqrt{x_{4}^{2} - r^2}}\frac{\partial}{\partial x_4}+\frac{r}{\sqrt{x_{4}^{2} - r^2}}\frac{\partial}{\partial r}.
\end{equation} 
Moreover, (\ref{HBe}) reduces to an ODE: 
$\left(f^{\prime \prime} f^2 + 2(f^{\prime})^2 f \right) \zeta(1-\zeta^2)-2f^{\prime}f^2(1+\zeta^2)=0$ with the solution ($C$ is an integration constant): 
$$f(\zeta) = C\left(\frac{2\zeta}{1-\zeta^{2}}+\ln\frac{1-\zeta}{1+\zeta}\right)^{\frac{1}{3}}.$$
The corresponding field $\varphi$ is therefore $\sigma_3$-critical. As in this case $\lambda_{1}^{2}\lambda_{2}^{2}\lambda_{3}^{2} = (x_{4}^{2} - r^2)^{-3}$, the associated stiff perfect fluid is $(U, p=(x_{4}^{2} - r^2)^{-3}, \rho = p)$, cf. also \cite{cso}. Recall that the energy conservation along flow lines is satisfied by construction as \eqref{titi} is true for any submersion.

Since $U$ is a standard conformal vector field, it is possible to construct a horizontally conformal submersion with fibers tangent to $U$. A way to do this is to firstly impose horizontal conformality (i.e. $\lambda_{1}^{2}=\lambda_{2}^{2}$) and then to employ Remark 3.1 $(e)$ to render $\varphi$ $\sigma_3$-critical.
We illustrate this strategy for a co-rotational submersion to the 3-sphere $\psi : \mathbb{M}^4 \to (\Ss^3, can)$ given by 
$$\left(x_4, \ r(\cos s , \ \sin s \cdot e^{i\theta})\right) \mapsto
\left(\cos \beta(x_4, r), \ \sin \beta(x_4, r) \, (\cos s , \ \sin s \cdot e^{i \theta})\right),$$
depending on a function $\beta : \RR \times (0, \infty) \to [0, \pi]$. 
Restricting again to an ansatz $\beta(x_4, r)= f(r/x_4)$ in the interior of the forward lightcone and asking $\varphi$ to be horizontally conformal (i.e. $\lambda_{1}^{2}=\lambda_{2}^{2}$), we are leaded to an ODE with the explicit solution $f=\arcsin$. 
In this case: $\lambda_{1}^{2}=\lambda_{2}^{2}=\lambda_{3}^{2}=x_{4}^{-2}$ and after a conformal rescale of the $\Ss^3$ metric $\widetilde h = (\cos^{-2}\beta) \cdot \textit{can}$, we shall obtain (via Remark 3.1 $(e)$) a $\sigma_3$-critical map  into a hemisphere, with the same associated perfect fluid. 
\end{ex}

Let us comment some particular features of the above solution:

\noindent $\bullet$ the flow is not accelerating. That is, the associated submersion have minimal fibers ($\nabla_U U = 0$) and the transversal "volume" $\lambda_{1}^{2}\lambda_{2}^{2}\lambda_{3}^{2}$ is constant in horizontal directions;

\noindent $\bullet$ the flow vector $U$ is \textit{shear-free}, i.e. $\mathcal{L}_U g^{\HH} - \frac{2}{3} \mathrm{div}U \cdot g^{\HH}=0$;

\noindent $\bullet$ the \textit{bulk viscosity} of the fluid is divergence free, i.e. $\mathrm{div}(\mathrm{div}U \cdot g^{\HH})=0$;

\noindent $\bullet$  the \textit{heat-flow vector} $Q=\gr^{\HH}T + T\nabla_U U$, for the standard temperature profile $T=(x_{4}^{2} - r^2)^{-\frac{3}{2}}$ is zero.
Together with the above two properties this means that this extension of Hwa-Bjorken flow is also a solution of the relativistic Navier-Stokes equations (see Appendix).

\begin{ex}[\textbf{Extended Hwa-Bjorken flow (II)}]\label{bjo2}
Let $\RR^3$ be the Euclidean space parametrized by cylindrical coordinates $(x_3 , x_{\perp}\cdot e^{\ii \phi})$ and regard $\mathbb{M}^4$ as $\RR \times \RR^3$. Consider $\varphi : \mathbb{M}^4 \to \RR^3$, an $SO(2)$-equivariant submersion of the following form:
\begin{equation}
\left(x_4, \ x_3 , \ x_{\perp}\cdot e^{\ii \phi} \right) \mapsto
\left(\beta(x_4, x_3), \ x_{\perp}\cdot e^{\ii \phi} \right),
\end{equation}
depending on a function $\alpha : \RR \times [0, \infty) \to \RR$.
Restricting $\varphi$ to the future wedge $x_4 > \abs{x_3}$ and choosing the ansatz $\beta(x_3, x_4)=f(x_3 / x_4)$ we can proceed as in the previous example to find the solution $f(x_3 / x_4)= C \, \mathrm{arctanh}(x_3 / x_4)$ ($C$ is an integration constant) for the reduced equation, solution that provide us with a $\sigma_3$-critical submersion. The corresponding stiff perfect fluid is then given by:
\begin{equation}\label{HBflow2}
U=\frac{x_4}{\sqrt{x_{4}^{2} - x_{3}^{2}}}\frac{\partial}{\partial x_4}+\frac{x_{3}}{\sqrt{x_{4}^{2} - x_{3}^{2}}}\frac{\partial}{\partial x_{3}}, \quad \rho = p = \frac{1}{x_{4}^{2} - x_{3}^{2}}
\end{equation}
\end{ex}
Notice that the flow (\ref{HBflow2}) is no more shear-free and this solution does not trivially satisfy the  relativistic Navier-Stokes equations as the previous one did. Nevertheless in \cite{gub} a correction of the energy density has been given such that the same flow profile, with radiation EoS $\rho = 3p$, becomes an exact solution for the relativistic Navier-Stokes equations with $\chi=\zeta=0$ and non-zero shear viscosity coefficient.

\subsection{The $\sigma_{3}^{k}$-critical field / cosmological perfect fluid correspondence}

In what follows we present the generalization of the above approach to the (linear) barotropic EoS: $p=(\gamma-1)\rho$, where $\gamma$ is a constant. When $1 \leq \gamma \leq 2$, this EoS is called \textit{cosmological} \cite{choq}.

Consider the following energy integral of a submersive field 
$\varphi : (M^4, g) \to (N^3,h)$ from a spacetime to some 3-dimensional Riemannian manifold:
\begin{equation}\label{Ebaro}
\mathcal{E}_{\sigma_{3}^{k}}(\varphi)= \int_{M} \abs{\wedge^3 \dif \varphi}^{2k} \nu_g, \quad k > 0.
\end{equation}

\begin{te}\label{teo}
The stress-energy tensor associated to (\ref{Ebaro}) is given by:
\begin{equation}\label{s3pfgen}
S_{\sigma_{3}^{k}}(\varphi) = -\frac{2k-1}{2} \lambda_{1}^{2k}\lambda_{2}^{2k}\lambda_{3}^{2k} g^{\HH} - \frac{1}{2} \lambda_{1}^{2k}\lambda_{2}^{2k}\lambda_{3}^{2k} g^{\VV},
\end{equation}
\noindent and the Euler-Lagrange equations associated to (\ref{Ebaro}) are equivalent to:
\begin{equation}\label{ELbaro}
-(2k-1)\gr^{\HH}(\ln \lambda_{1}\lambda_{2}\lambda_{3}) + \mu^{\VV}=0.
\end{equation} 
\end{te}

\begin{proof}
Use again (\ref{stres3}) to compute 
$$\frac{\dif}{\dif u} \Big| _{u=0}\mathcal{E}_{\sigma_{3}^{k}}(\varphi, g_u)=\int_M \langle \frac{1}{2} \sigma_{3}^{k}(\varphi, g)\cdot g - k\sigma_{3}^{k-1}(\varphi, g) \cdot \varphi^* h \circ \chi_2(\varphi , g),\ \delta g \rangle \nu_g.$$
In order to conclude, perform a similar computation as in the proof of Proposition \ref{stif} $(ii)$, then take the divergence of $S_{\sigma_{3}^{k}}(\varphi)$ on a horizontal vector.
\end{proof}

\begin{co}
Let $\varphi: (M^4, g) \to (N^3, h)$ be a submersion with timelike fibers tangent to the the unit vector $U$. Then the variational equations (\ref{ELbaro}) together with the identity \eqref{titi} represent, respectively, the relativistic Euler equation and the conservation of energy along flow lines for the perfect fluid $(U, \rho, p)$ with 
$\rho = \lambda_{1}^{2k}\lambda_{2}^{2k}\lambda_{3}^{2k}$
and the cosmological EoS $p=(2k-1)\rho$.
\end{co}

In particular, we obtain a radiation EoS when $k=2/3$ (when the 
corresponding $\sigma_{3}^{k=2/3}$ - variational problem is conformally invariant), the cosmological constant EoS $\rho=-p$ when $k=0$ (identically constant Lagrangian) and the \textit{dust} EoS $p=0$ when $k=1/2$.

To illustrate the above results, simply notice that the submersions constructed in Examples \ref{bjo1} and \ref{bjo2} are $\sigma_{3}^{k}$-critical for all values of $k$. Accordingly, the perfect fluids corresponding to Example \ref{bjo1} are $(U, p=(x_{4}^{2} - r^2)^{-3k}, \rho=(2k-1)^{-1}p)$ with $U$ given by (\ref{HBflow}).

\medskip

Let us rewrite these solutions in the customary coordinates for heavy-ion collisions: 
$$
\tau = \sqrt{x_{4}^{2} - x_{3}^{2}}, \quad 
\eta = \mathrm{arctanh}\frac{x_3}{x_4}, \quad
x_{\perp} = \sqrt{x_{1}^{2} + x_{2}^{2}}, \quad
\phi = \arctan \frac{x_2}{x_1},
$$
where the Minkowski metric reads: $g=-\dif \tau^2 + \tau^2 \dif \eta^2+\dif x_{\perp}^{2} + x_{\perp}^{2}\dif \phi^2$. 

\begin{ex}\label{gubsnot}
The radiation type perfect fluid analogous to Example \ref{bjo1} is
\begin{equation}
U=\frac{\tau}{\sqrt{\tau^{2} - x_{\perp}^{2}}}\frac{\partial}{\partial \tau}+\frac{x_{\perp}}{\sqrt{\tau^{2} - x_{\perp}^{2}}}\frac{\partial}{\partial x_{\perp}}, \quad \rho = 3p = \frac{1}{(\tau^{2} - x_{\perp}^{2})^2}.
\end{equation}
Corresponding to Example \ref{bjo2}, we have the radiation type perfect fluid:
\begin{equation}\label{gubs0}
U=\frac{\partial}{\partial \tau}, \quad \rho = 3p = \frac{1}{\tau^{\frac{4}{3}}}.
\end{equation}
\end{ex}
A generalization of the perfect fluid (\ref{gubs0}) is given in \cite{gub, gubs}. It corresponds to the $SO(3)$-equivariant, $\sigma_{3}^{k}$-critical (for all $k$) submersion $\varphi: dS_3 \times \RR \to \mathbb{S}^2 \times \RR$,
$$
(\sinh \varrho , \ \cosh \varrho (\cos s, \sin s \cdot e^{\ii \phi}), \ \eta) \mapsto ((\cos s, \ \sin s \cdot e^{\ii \phi}), \ \eta),$$
where $dS_3$ is the de Sitter 3-dimensional space, $\mathbb{S}^2$ is the two-sphere and both domain and codomain are endowed with the product metric. The unit vertical vector is $U=\partial / \partial \varrho$ and the eigenvalues of $\Cc_\varphi$ are $\lambda_{1}^{2}=1$, $\lambda_{2}^{2}=\lambda_{3}^{2}=\cosh^{-2}\varrho$ ; thus  $\varphi$ induces in particular the radiation type fluid $(U, \rho=3p=\cosh^{-8/3}\varrho)$ corresponding to $\sigma_{3}^{k=2/3}$-variational problem.\\
As $dS_3 \times \RR$ is conformally equivalent to the Minkowski space and $\sigma_{3}^{k=2/3}$-energy is conformally invariant, the above submersion produces a $\sigma_{3}^{k=2/3}$-critical map (so a radiation type perfect fluid) on the (future wedge in the) Minkowski space.

\begin{re}
More general EoS can be treated analogously, by considering an energy  of the type $\int_M F(\abs{\wedge^3 \dif \varphi}^{2}) \nu_g$ (see also \cite{bei, cristo, langl, ratta}).
\end{re}

\subsection{$r$-Harmonic morphisms / shear-free barotropic fluids correspondence}

In the following we extend Remark \ref{remstif} $(d)$ by showing that the $r$-harmonic morphisms are related to shear-free perfect fluids with various EoS. Recall \cite{lube} that a ($r$-)harmonic morphism is a ($r$-)harmonic and horizontally conformal map. For examples and more details on harmonic morphisms see \cite{ud, baud, fug}.

\begin{pr}
Let $\varphi : (M^{4}, g) \to (N^3, h)$ be a submersion  with timelike fibers from a spacetime to some Riemannian manifold and $U$ an unit vertical timelike vector field. If $\varphi$ is horizontally conformal, i.e. $\lambda_{1}^{2}=\lambda_{2}^{2}=\lambda_{3}^{2}$, then

\medskip
$(i)$ \  $U$ is a shear-free vector field,

\medskip
$(ii)$ \ $\varphi$ is $\sigma_{3}^{k}$-critical if only if it is a $6k$-harmonic morphism. If this is the case, then 
$(U, p = (2k-1)\lambda^{6k}, \rho = \lambda^{6k})$ is a shear-free perfect fluid with the EoS $p= (2k-1)\rho$.
\end{pr}

\begin{proof}
$(i)$ As $\varphi$ is horizontally conformal, we have $\lambda^2 g^{\HH} =\varphi^* h$ and therefore $\mathcal{L}_{U} \left(\lambda^2 g^{\HH}\right)=0$ that gives us the result.

\noindent $(ii)$ For a harmonic morphism on a $m$-dimensional spacetime, the fundamental equations hold \cite{sig}:
$(n-2)\gr^{\HH}(\ln \lambda) + (m-n)\mu^{\VV}=0$. They can be generalized for $r$-harmonic morphisms as follows:
\begin{equation*}
(n-r)\gr^{\HH}(\ln \lambda) + (m-n)\mu^{\VV}=0,
\end{equation*}
 where $m=4$ and $n=3$ in our case. The conclusion easily follows.
\end{proof}

For instance, if $\varphi$ is a harmonic morphism with dilation $\lambda^2$, then $(U, p = -\frac{1}{3}\lambda^2, \rho = \lambda^2)$ is a shear-free perfect fluid with the EoS $\rho = -3p$ (\emph{quintessence}). Analogously, if $\varphi$ is a $4$-harmonic morphism, then $(U, p = \frac{1}{3}\lambda^4, \rho = \lambda^4)$ is a shear-free perfect fluid with the EoS $\rho = 3p$ (\emph{radiation}).

We emphasize the fact that locally there is a one to one correspondence between horizontally conformal maps and shear-free (vertical) vector fields,
cf. \cite[Prop. 2.5.11]{ud}. According to the previous Proposition this correspondence holds also between $r$-harmonic morphisms and shear-free linear-barotropic fluid flows (general barotropic flows can be treated analogously by employing the notion of $F$-harmonic morphism \cite{ou}). Now recall that harmonic morphisms with one dimensional fibers defined on Riemannian manifolds are subject to various classification results \cite{bry, pan}. Similar results for $r$-harmonic morphisms in Lorentzian signature (with the condition of timelike fibers) would locally classify the associated shear-free perfect fluids. The main result in \cite{te} is an example of such classification, for radiative shear-free fluids that are coupled with gravitation (see next section).

\begin{ex}[\textbf{Morawetz flow}]
For the submersions defined by \eqref{bjsubm} let us consider the $\sigma_{3}^{2/3}$-Euler-Lagrange equations (\ref{ELbaro}) that reduce to one single PDE analogous to (\ref{HBe}):
\begin{equation}\label{radiationeq}
\cosh \Omega \left(\frac{\partial \Omega}{\partial x_4} + \frac{1}{3}\frac{\partial (\ln \lambda_{1}\lambda_{2}\lambda_{3})}{\partial r} \right) + \sinh \Omega \left(\frac{\partial \Omega}{\partial r} + \frac{1}{3}\frac{\partial (\ln \lambda_{1} \lambda_{2} \lambda_{3})}{\partial x_4} \right)=0.
\end{equation}
This time we impose the ansatz $\alpha(x_4, r)=f(\ln(\frac{x_{4}^{2}-r^2}{r}))$ defined on the forward lightcone interior. The vertical timelike vector of these submersions will be the (normalized) Morawetz vector field \cite{mor}:
\begin{equation}\label{MORflow}
U=\frac{x_{4}^{2} + r^2}{x_{4}^{2} - r^2}\frac{\partial}{\partial x_4}+\frac{2 x_{4} r}{x_{4}^{2} - r^2}\frac{\partial}{\partial r}.
\end{equation}
In this case $\lambda_{1}^{2}=(f^{\prime})^2 / r^2$
and $\lambda_{2}^{2}=\lambda_{3}^{2}=f^2 / r^2$ and (\ref{radiationeq}) reduces again to an ODE: $(\ln f^{\prime}f^2)^{\prime}=-3$. The corresponding solution $\alpha(x_4, r)=\frac{r}{x_{4}^{2}-r^2}$ provides us with a horizontally conformal $\sigma_{3}^{2/3}$-critical map, that is with a 4-harmonic morphism. Computing $\rho$ $=(\lambda_{1}\lambda_{2}\lambda_{3})^{4/3}=\lambda^{4}=(x_{4}^{2}-r^2)^{-4}$, we get the associated radiative perfect fluid $(U, \rho, p=\frac{1}{3}\rho)$. This shear-free and accelerated (i.e. $\nabla_U U \neq 0$) flow was pointed out in \cite{nag} where other exact solutions with $\Omega(x_4, r)= \kappa \cdot \mathrm{arctanh}(r/x_{4})$, $\kappa \in \RR$ can be found. Recall that the Morawetz vector field belongs to the family of \textit{special conformal transformations} so this example may admit further generalizations. 
\end{ex}

\begin{ex}[\textbf{Skew projection's flow}]
Consider the cylinders $$\mathcal{C}^{4}_{1}(q) = \{x= (x_4, (x_3 , \  x_{\perp} \cdot e^{\ii \phi})) \in \mathbb{R}^{4}_{1} \ \vert \ x_{\perp}^{2} < q^{-2}\}, \quad \mathcal{C}^{3}(q) = \{\tilde x =(\tilde{x}_{3} , \tilde{x}_{\perp} \cdot e^{\ii \tilde \phi}) \in \mathbb{R}^3 \ \vert \ \tilde{x}_{\perp}^{2} < q^{-2}\}$$ in Minkowski and 3-Euclidean space, respectively. On the first consider the induced Lorentzian metric and on the second consider the deformed metric
$\tilde h =(1 - q^2 \tilde{x}_{\perp}^2)\cdot[\dif \tilde{x}_3^{2} + \dif \tilde{x}_{\perp}^{2}]+ \tilde{x}_{\perp}^{2}\dif \tilde{\phi}^2$. Let $\varphi : \mathcal{C}^{4}_{1}(q) \to \mathcal{C}^{3}(q)$ be the axially symmetric submersion given by:
\begin{equation}
\left(x_4, \ (x_3 , \  x_{\perp} \cdot e^{\ii \phi})\right) \mapsto
\left(x_{3} , \ x_{\perp} \cdot e^{\ii (\phi - q x_4)}\right).
\end{equation}
We can easily verify that $U= \frac{1}{\sqrt{1 - q^2 x_{\perp}^2}} \frac{\partial}{\partial x_4} + \frac{q}{\sqrt{1 - q^2 x_{\perp}^2 }} \frac{\partial}{\partial \phi}$ is the unit timelike vertical vector field and the eigenvalues of the Cauchy-Green tensor of $\varphi$ are $\lambda_{1}^{2} = \lambda_{2}^{2} = \lambda_{3}^{2} = 1 - q^2 x_{\perp}^2 :=\lambda^{2} $ and
$\lambda_{4}^{2} = 0$.
Now let us notice that the following identities hold:
\begin{equation}\label{skeweq}
\nabla_U U - \gr^{\HH}(\ln \lambda)=0, \quad
\di U + U(\ln \lambda^3)=0 
\end{equation}
where the second is trivially satisfied as $U$ is divergence free.  Then (\ref{skeweq}) tells us simultaneously that $\left(U, \ p = - \rho / 3, \ \rho=\lambda^2\right)$ satifies the equations of a (shear-free) perfect fluid and that the associated submersion $\varphi$ is a harmonic morphism.
\end{ex}

\begin{ex}[\textbf{Radiation flow in Einstein universe}]
Consider the static spacetime $(\RR \times \Ss^3, g)$ parametrized as $(t, \cos s \cdot e^{\ii \phi_1}, \sin s \cdot e^{\ii \phi_2})$ so that the product Lorentzian metric reads $g=-\dif t^2 + \dif s^2 + \cos^{2}s \, \dif \phi_{1}^{2}+\sin^2 s \, \dif \phi_{2}^{2}$.
Take $\omega_1,\omega_2 <1$ and define the axially symmetric submersion
$\varphi_{\omega}: \RR \times \Ss^3 \to \Ss^3$ by
$$
(t, \cos s \cdot e^{\ii \phi_1}, \sin s \cdot e^{\ii \phi_2})
\mapsto
(\cos s \cdot e^{\ii (\phi_1 - \omega_1 t)}, \sin s \cdot e^{\ii (\phi_2- \omega_2 t)}).
$$
whose unit timelike vertical vector field is:
$$
U=\frac{1}{\sqrt{1-\omega_{1}^{2}\cos^2 s-\omega_{2}^{2}\sin^2 s}}\left(\frac{\partial}{\partial t}+\omega_1 \frac{\partial}{\partial \phi_1}+\omega_2 \frac{\partial}{\partial \phi_2}\right).
$$
Write the standard metric on $\Ss^3$ as $h= \sum_{i=1}^{3} e^i \otimes e^i$ with $e^i$'s given by: 

$e^1 = \dif s$, 

$e^2 = (\omega_{1}^{2}\cos^2 s + \omega_{2}^{2}\sin^2 s)^{-1/2}(\omega_{1}\cos^2 s \, \dif \phi_1+\omega_{2}\sin^2 s \, \dif \phi_2)$ and
 
$e^3 = \sin s \cos s (\omega_{1}^{2}\cos^2 s + \omega_{2}^{2}\sin^2 s)^{-1/2}
(\omega_{2} \, \dif \phi_1-\omega_{1} \, \dif \phi_2)$.

Now consider the following non-standard metric on the codomain:\\
$\tilde{h}= A(s)^{-2}[A(s)e^1 \otimes e^1+e^2 \otimes e^2+A(s)e^3 \otimes e^3]$, where $A(s):=1-\omega_{1}^{2}\cos^2 s-\omega_{2}^{2}\sin^2 s$. Then the eigenvalues for $\Cc_\varphi$ with respect to $\tilde{h}$  are: $\lambda_{1}^{2}=\lambda_{2}^{2}= \lambda_{3}^{2}=A(s)^{-1}:=\lambda^{2}$, $\lambda_{4}^{2}=0$.

We can check that $\varphi_\omega$ is a 4-harmonic morphism (so $\sigma_{3}^{2/3}$-critical submersion), and therefore $U$ is a shear-free perfect fluid with the energy density $\rho=(1-\omega_{1}^{2}\cos^2 s-\omega_{2}^{2}\sin^2 s)^{-2}$ and $p=\frac{1}{3}\rho$. This solution has been analysed in \cite[$\S$3]{bha} where it is shown that it also satisfies the relativistic Navier-Stokes equations for an appropriate temperature profile.
\end{ex}

\subsection{Irrotational flows}

In the following we characterize irrotational flows in terms of their associated  field.

Suppose that the submersive field with timelike fibers $\varphi : (M^4, g) \to (N^3,h)$ is a critical point for the functional $\mathcal{E}_{\sigma_{3}^{k}}$. Then, according to Corollary 3.1, the timelike unit vector $U$ that spans the fibres is the flow vector field of a perfect fluid. Let  $V = (1/f)U$ be the \textit{fundamental vector field} of $\varphi$, where $f$ is the index of the fluid. Denote its dual 1-form with $\vartheta$. The next result and its proof are analogous to the Riemannian case, see \cite[p. 341]{ud}.

\begin{pr} 
The 2-form $\dif \vartheta$ identically vanishes if and only if the horizontal distribution is integrable, i.e. $U$ is orthogonal to a (spacelike) hypersurface. 
\end{pr}

So $\dif \vartheta$ will be again called the \textit{integrability $2$-form} of $\varphi$. Notice that $V$ and $\dif \vartheta$ play a r\^ole similar to the \textit{dynamic $4$-velocity} and the \textit{vorticity tensor} \cite{choq}.

An easy check shows us that the horizontal distribution is integrable for Examples \ref{bjo1} and \ref{bjo2}, so $U$ is irrotational in both cases.

\section{Coupling with gravity}

Let $(M,g)$ be a 4-dimensional spacetime, with $\mathrm{Ric}^M$ and $\mathrm{Scal}^M$ denoting its Ricci and scalar curvatures, respectively. We call \textit{perfect fluid coupled with gravity} a triple $(U, \rho, p)$ that solves \textit{the Einstein equation}: 
\begin{equation}\label{ein}
\mathrm{Ric}^M - \frac{1}{2}\mathrm{Scal}^M \cdot g = \alpha \Tt,
\end{equation} 
where $\alpha$ is constant and $\Tt$ is the fluid stress-energy tensor given by (\ref{perfstress}). The existence of such solution impose severe restrictions on the curvature as shown by the following result (see e.g. \cite{one}).

\begin{pr}
Let $U$ be a future-pointing timelike unit vector field on the spacetime $M$. 
Then $U$ is the flow vector field of a perfect fluid coupled with gravity if and only if $\mathrm{Ric}^M(X, Y)$ is proportional to $g(X, Y)$ and $\mathrm{Ric}^M(X, U) = 0$ for all $X$ and $Y$ orthogonal to $U$. In this case the energy density and the pressure of the fluid are:
$$
\rho=\frac{1}{\alpha}\left[\mathrm{Ric}^M(U,U) + \frac{1}{2}\mathrm{Scal}^M \right], \quad
p=\frac{1}{3\alpha}\left[\mathrm{Ric}^M(U,U) - \frac{1}{2}\mathrm{Scal}^M \right]
$$
\end{pr}

A standard case when the above conditions on Ricci tensor are satisfied is provided by the warped product $(M = I \times_{f} N, g=-\dif t^2 + f^2 \cdot h)$ where $(N, h)$ is a (Riemannian) Einstein manifold and $f$ a function defined on the real interval $I$. 

\medskip

Turning to the variational picture, let $\varphi: (M^4, g) \to (N^3, h)$ be a smooth map from a Lorentzian manifold to a Riemannian 3-manifold. Consider a
common energy $\mathcal{E}(g,\varphi)$ for the field $\varphi$ coupled to gravity given as
\begin{equation}
\mathcal{E}(g ,\varphi)= \mathcal{E}_{grav}(g) - \alpha \mathcal{E}_{field}(g ,\varphi),
\end{equation}
where $\mathcal{E}_{grav}(g) = \int_{M}\mathrm{Scal}^M \nu_g$, $\mathcal{E}_{field}(g ,\varphi) = \int_{M}\abs{\wedge^3 \dif \varphi}^{2k}\nu_g$ and $\alpha$ is a coupling constant. For a similar coupling of the standard kinetic term with gravity, see e.g. \cite{heu, mi, mus}.

Then the stationary points of $\mathcal{E}(g ,\varphi)$ with respect to smooth variations of $g$ and $\varphi$, on any
compact domain $K \subset M$, are characterized by:
\begin{equation}\label{couple}
\left\{
\begin{array}{ccc}
\mathrm{Ric}^M - \frac{1}{2}\mathrm{Scal}^M \cdot g &=& \alpha S_{\sigma_{3}^{k}}(\varphi); \\[2mm]
\tau_{\sigma_{3}^{k}}(\varphi) &=& 0,
\end{array}
\right.
\end{equation}
where $S_{\sigma_{3}^{k}}(\varphi)$ and $\tau_{\sigma_{3}^{k}}(\varphi)$ are the stress-energy tensor and Euler-Lagrange operator associated to $\mathcal{E}_{field}$.

According to Theorem \ref{teo}, if a submersion with timelike fibers $\varphi$ satisfies the first equation in \eqref{couple} then it will provide us with a (linear) barotropic fluid coupled to gravitation. Notice that the second equation in \eqref{couple} will be a consequence of the first one. 

Let us give a simple example. More examples of $\sigma_{3}^{k}$-critical submersions coupled with gravity can be obtained from the perfect fluid solutions of the Einstein equations discussed in \cite{steph}. 

\begin{ex}
Let us consider the Robertson-Walker spacetime $M_{k, f} = I \times_{f} N$, where $(N,h)$ is a 3-dimensional, connected Riemannian manifold of constant curvature $k$. Then $U=\partial / \partial t$ is the flow vector of a barotropic perfect fluid coupled with gravity, with $\rho$ and $p$ depending only on the time $t \in I$ \cite{one} (in fact it solves the full relativistic Navier-Stokes equations as the dissipative terms of the stress-energy tensor vanish in this case). 

If we take $f(t)=at+b$ an affine function, then $\rho = 3(a^2 + k)(at+b)^{-2}$ and the EoS for this fluid is: $p=-\frac{1}{3}\rho$. The associated $\sigma_{3}^{k=1/3}$-critical submersion is the projection on the second factor $\pi_2: (M_{k, f}, -\dif t^2 + f^2 (t) h) \to (N, 3(a^2 + k)\cdot h)$; it is, in particular, a harmonic morphism (horizontally homothetic with totally geodesic fibers) and a solution of (\ref{couple}) for $k=1/3$ and $\alpha=-2$.

If we take $f(t)=\sqrt{2at-kt^2}$, $a>0$ (defined only for finite time when $k>0$), then $\rho = 3a^2 f(t)^{-4}$ and the EoS for this fluid is: $p=\frac{1}{3}\rho$. The associated $\sigma_{3}^{k=2/3}$-critical submersion is again the projection $\pi_2: (M_{k, f}, -\dif t^2 + f^2 (t) h) \to (N, \ a\sqrt{3}\cdot h)$; it is also a 4-harmonic morphism and a solution of (\ref{couple}) for $k=2/3$ and $\alpha=-2$. 
\end{ex}

\section{Appendix: Thermodynamic degrees of freedom}
Recall that generally a relativistic fluid is characterized by five thermodynamic scalars that are subject to the \textit{first law of thermodynamics} (\textit{Gibbs equation}):
\begin{equation}\label{gib}
\dif \rho = \frac{\rho + p}{n}\dif n + nT\dif s.
\end{equation}
Here $n$ is the \textit{particle number density} with the associated conservation law $\di (nU)=0$, $s$ is the \textit{specific entropy} and $T$ is the \textit{temperature}.
Gibbs equation shows that in general two of these five scalars are needed as independent variables. For example, taking $n$ and $s$ as independent, the remaining thermodynamical scalars are $p(n, s)$, $\rho(n, s)$, $T(n, s)$, and given any one of these, say $\rho = \rho(n, s)$, the others will be determined by $p=-\rho + n\frac{\partial \rho}{\partial n}$ and $T=\frac{1}{n} \cdot \frac{\partial \rho}{\partial s}$.

In our context $n=\lambda_1 \lambda_2 \lambda_3$ (so that the conservation law $\di (nU)=0$ is automatically satisfied according to Lemma \ref{diu}) and $\rho=F(n^2)$, so $p=-F(n^2)+2n^2F^{\prime}(n^2)$. Therefore a fluid that comes from a sigma model will always have a barotropic EoS for the energy density and the pressure. It can be easily shown that in this case the fluid is \textit{isoentropic} (i.e. $s$ is constant), while a general perfect fluid is only \textit{adiabatic} (i.e. $U(s)=0$). It would be interesting to find an extension of the action principle used in this paper that account for entropy production (at least by shear and bulk viscosities).

Recall that the dissipative fluids (or \textit{imperfect} fluids) dynamics is modelled by the \textit{relativistic Navier-Stokes equations} $\di \Tt =0$,  where $\Tt$ is the \textit{modified} stress-energy tensor (\cite[p.567]{mis}, \cite[p.56]{wei}) given by:
\begin{equation*}\label{ns}
\Tt=p \, g + (p + \rho)\omega \otimes \omega 
- \eta \left(\mathcal{L}_U g^\HH - \frac{2}{3}\di U \cdot g^\HH \right)  - \chi(\theta \otimes \omega + \omega \otimes \theta) - \zeta \di U \cdot g^\HH.
\end{equation*}
where $\theta = (\gr^{\HH} T + T \nabla_{U}U)^{\flat}$ is the dual of the \textit{heat-flow vector} and $\eta$, $\chi$, $\zeta$ are the \textit{shear viscosity}, \textit{heat conduction} and \textit{bulk viscosity} coefficients, respectively (all are scalars that may depend on the temperature and the mean free time). The above stress-energy choice is motivated by the fact that together with \eqref{gib} it assures a positive rate of entropy generation: $\di (nsU - \frac{\chi}{T}\theta^{\sharp}) \geq 0$ (\textit{the second law of thermodynamics}).

\bigskip

\begin{acknowledgments}
I thank C\u alin Alexa (IFIN-HH, Bucharest and CERN, Geneva) for introducing me this subject and for helpful discussions, Gabriel B\u adi\c toiu (IMAR, Bucharest) for many comments on the preliminary manuscript and Willie Wong (University of Cambridge) for pointing me to the reference \cite{cristo}. I am also grateful to Professor Tudor Ra\c tiu and the Department of Mathematics at \emph{Ecole Polytechnique F\'ed\'erale de Lausanne} for hospitality during the preparation of the present paper. This research was supported by \emph{PN II Idei Grant, CNCSIS, code 1193}.
\end{acknowledgments}

\label{lastpage}
\end{document}